\documentclass[english,a4paper]{amsart}

\usepackage{amssymb}
 \usepackage{amsthm}

\usepackage{lineno}

\usepackage{amsmath}

\usepackage{mathrsfs,amsfonts,dsfont} 
\usepackage{mathrsfs}
\usepackage{bbm}
\usepackage{amsfonts}
 \usepackage{dsfont}
 \usepackage{graphicx}
\usepackage{mathrsfs}
\usepackage{wrapfig}
\usepackage{tikz}
\usetikzlibrary{arrows,positioning,shapes.geometric}
\usepackage{amsrefs}
\usepackage{babel}
\usepackage[latin1]{inputenc}

 \newcommand{\vero}[1]{{\color{black}#1}}

\def\bE{\mathbb E}
\def\E{\mathbb E}
\def\bN{\mathbb N}
\def\N{\mathbb N}
\def\bP{\mathbb P}
\def\bR{\mathbb R}
\def\R{\mathbb R}
\def\ss{\mathbb S}

\def\cA{{\mathcal A}}

\def\UU{{\mathcal U}}

\def\X {X:=\{X(s),s\in\ss\}}
\def\Y {Y:=\{Y(s),s\in\ss\}}

\DeclareMathOperator*{\argmin}{argmin}

\newtheorem{Th}{Theorem}[section]

\newtheorem{Pro}[Th]{Proposition}

\newtheorem{remark}{Remark}[section]

\numberwithin{equation}{section}

\title{A semi-parametric estimation for max-mixture spatial processes}
\author{M. Ahmed}
\address{Universit\'e de Lyon, Universit\'e Lyon 1, Institut Camille Jordan ICJ UMR 5208 CNRS, France\\
Department of statistics, University of Mosul, Iraq}
\author{ V. Maume-Deschamps}
\address{Universit\'e de Lyon, Universit\'e Lyon 1, Institut Camille Jordan ICJ UMR 5208 CNRS, France}
\author{P.Ribereau}
\address{Universit\'e de Lyon, Universit\'e Lyon 1, Institut Camille Jordan ICJ UMR 5208 CNRS, France}
\author{C.Vial}
\address{Universit\'e de Lyon, Universit\'e Lyon 1, Institut Camille Jordan ICJ UMR 5208 CNRS, France}

\keywords{Spatial dependence measures, asymptotic dependence/independence, max-stable process, max-mixture, madogram}

\begin{document}




\begin{abstract}
We propose a semi-parametric estimation procedure in order to estimate the parameters of a max-mixture model as an alternative to composite likelihood \vero{estimation}. \vero{This procedure uses} 
the  \textit{F-madogram}. \vero{We propose to minimize} the square difference between the theoretical F-madogram and an empirical one. We evaluate the performance of this estimator through a 
simulation study \vero{and we compare our method to composite likelihood estimation}. We apply our estimation procedure to daily rainfall \vero{data from} East Australia.  \\
\end{abstract}

\maketitle




\section{ Introduction}
\vero{One of the main characteristic of environmental or climatic data is their } spatial dependence. The dependencies may be strong even for large distances as \vero{observed for} heat waves or 
they may be strong at short distances and weak at larger distances, as \vero{observed for cevenol rainfall} events. Many dependence structures may arise, for example, asymptotic 
dependence, asymptotic independence, or both. \vero{Max-mixture models as defined in \cite{wadsworth2012dependence}  are a mixture between a max-stable process and an Asymptotic Independent (AI) 
process. These kind of models may be useful to fit e.g. rainfall data (see \cite{bacro2016flexible})}.  \\
 \\
  The estimation of the parameters of these processes remains \vero{challenging}. The usual way to estimate parameters in spatial contexts is to maximize the composite likelihood. For 
example, in \cite{padoan2010likelihood}, \cite{davison2012geostatistics} and many others, the composite likelihood maximization is used to estimate the parameters of max-stable processes. In   
\cite{bacro2016flexible} and \cite{wadsworth2012dependence}, it is used to estimate the parameters of max-mixture processes. Nevertheless, the estimation \vero{remains unsatisfactory in some 
cases as it seems to have difficulties  estimating the AI} part. \vero{So that, an alternative should be welcomed.} 
\\
We propose a semi-parametric estimation procedure as an alternative  to composite likelihood maximization for max-mixture processes. Our procedure is  a least square method on the $F$-madogram. That 
is, we minimize the \vero{squared}  difference between the theoretical $F$-madogram and the empirical one. Some of  literature  deals with semi-parametric estimation in modeling spatial extremes. For 
example, \cite{northrop2015efficient} and \cite{aghakouchak2010semi} provided  semi-parametric estimators of extremal indexes. In  \cite{buhl2016Semiparametric}, another  semi-parametric  procedure to 
estimate model parameters is introduced. It is based \vero{a non linear least square method based on the extremogram } for isotropic space-time  Browen-Resnick max-stable processes. The 
semi-parametric procedure \vero{that we} propose in  this study is \vero{close to that} article.\\
\\
 Section \ref{Se:2} is dedicated to the main tools used in this study; it contains definitions of some spatial dependence measures  and  of different dependence structures (asymptotic 
dependence/independence and  mixture of them). In Section \ref{Se:3}, we calculate an expression for the $F$-madogram of max-mixture models. Section \ref{Se:4} is 
devoted to estimation procedures of  the parameters of max-mixture processes. \vero{We prove that  our least square-madogram estimation is consistent}, provided that the \vero{parameters}
are identified by the $F$-madogram.
\vero{In Section \ref{Se:5}, simulation study is conducted, it allows us to evaluatere the performance of the estimation procedure.} Section \ref{Se:6} is devoted {\vero to an application on} rainfall 
real data \vero{from} East Australia.  Finally, concluding remarks are discussed in Section \ref{Se:7}.\\
 
 \section{Main tools used in the study}\label{Se:2}
Throughout this study, the spatial process $\X, \ss\in\bR^d$ is assumed to be  strongly stationary and isotopic. \vero{We recall some basic facts and definitions related to spatial extreme theory.}
\subsection{Dependence measures}\label{DM:1}
\vero{The extremal dependence behavior of spatial processes may be described by several coefficient / measures. }\\
The \textbf{upper tail dependence coefficient} $\chi$ (see\cite{bacro2013measuring,sibuya1959bivariate}) is defined  for a stationary spatial process $X$ on $\ss\subset\bR^2$ with margin $F$ 
\vero{by:}
  \begin{equation}\label{max49}
\chi(h)=\lim_{u\to1^-}\bP\big(F(X(s+h))>u|F(X(s))>u\big)\/. 
\end{equation}
The process is said AI (resp. AD) if for all $h\in\ss$ $\chi(h)=0$  (resp. $\chi(h)\neq0$).

For any $h$ the coefficient $\chi(h)$ can alternatively be expressed as the limit when $u\to1^-$ of the function defined on $\ss\times[0,1]$ into $[0,1]$, by
 \begin{equation}\label{max51}
\chi(h,u)=2-\frac{\log\bP\big(F(X(s))<u,F(X(t))<u\big)}{\log\bP\big(F(X(s))<u\big)}, \textnormal{for } h\in\ss, u\in[0,1[.
\end{equation}
Such that, $\chi(h)=\lim_{u\to ^-1}\chi(h,u)$. \\
\\
\vero{For AI processes, $\chi$ cannot reveal the strength of the dependence.} This is why the authors in~\cite{coles1999dependence}, introduced an alternative dependance coefficient called 
\textbf{lower tail dependence coefficient $\overline{\chi}$}. 
\vero{Consider} for any $(s,s+h)\in\ss^2$
\begin{equation}\label{ind2}
\overline{\chi}(h,u)=\frac{2\log\bP\big(F(X(s))>u\big)}{\log\bP\big(F(X(s))>u,F(X(s+h))>u\big)}-1,\quad 0\leq u\leq1
\end{equation}  
 \vero{and} $\overline{\chi}(h)=\lim_{u\to1}\overline{\chi}(h,u)$.

 If $\overline{\chi}(h)=1$ for all $h$, the spatial process is asymptotically dependent. Otherwise, the process is  asymptotically independent. Furthermore, if $\overline\chi\in]0,1[$ ( resp. 
$]-1,0[$) the locations $s$ and $s+h$ (for any $s$) are asymptotically positively associated (resp. asymptotically negatively associated).

 Another important measure of dependence \vero{is the extremal coefficient which} was introduced by \cite{buishand1984bivariate,schlather2003dependence}.
\vero{For any $s\in\ss$ and $s+h\in\ss$ and $x\in\bR$, let}
 $$
 \theta_F(h,x)=\frac{\log(P(X(s)<x,X(s+h)<x))}{\log(P(X(s)<x))}.
 $$

\vero{$\theta_F$} is related  \vero{with} the upper tail dependence parameter; indeed if $\lim_{x\rightarrow x_F}\theta_F(h,x)=\theta(h)$ exists, we have the following relation 
(see \cite{bacro2013measuring}):
  $$\chi(h)=2-\theta_F(h),$$
 where $x_F=\sup\{x|F(x)<1\}$. \vero{In that case,} $P(X(s)<x,X(s+h)<x))$ may be approximated by $F(x)^{\theta_F(h)}$ for $x$ large.
 
 This coefficient is  particularly useful when dealing with asymptotic dependence, but useless in case of asymptotic independence. To overcome this problem, \cite{ledford1996statistics} proposed a 
model allowing to gather all the different dependence \vero{behaviors. This model has Fréchet marginal laws and} for all $(s,s+h)\in\ss^2$ the pairwise survivor  function \vero{is given by:}
 \begin{equation}\label{ind1}
 \bP\big(X(s)>x,X(s+h)>x\big)=\mathcal{L}_{h}(x)x^{-1/\eta(h)},\quad \textnormal{ as } x\to\infty\/,
\end{equation} 
where $\mathcal{L}_{h}$ is a slowly varying function and $\eta(h)\in(0,1]$ is the \textbf{tail dependence coefficient}. This coefficient  determines the decay rate of the bivariate tail 
probability for large $x$. The interest of this simple modelization, which appears to be quite general, is that the coefficient $\eta(h)$ provides a measure of the extremal dependence \vero{between} 
$X(s)$ and $X(x+h)$. \\
$\mathcal{L}(x)\not\to 0$ (resp.  $0<\eta(h)<1$ and $\mathcal{L}(x)\not\to 0$), corresponds to asymptotic dependence  (resp.  asymptotic independence); see~\cite{ 
bacro2013measuring,ledford1996statistics}. Finally, it is important to see the relation between $\eta$ and $\overline\chi$. If equation~(\ref{ind1}) is satisfied, then 
$\overline{\chi}(h)=2\eta(h)-1$.\\

Another classical tool often used in geostatistics is the variogram. But for spatial  processes with Fréchet \vero{marginal laws, the variogram does not exist.} We shall use the \textbf{$F$-madogram} 
introduced in \cite{cooley2006variograms} which is defined for any spatial process $X$ with univariate margin $F$ and  for all $(s,t)\in\ss^2$
\begin{equation}\label{mado}
\nu^F(s-t)=\frac{1}{2}\bE|F(X(s))-F(X(t))|.
\end{equation}

 \subsection{Spatial extreme models}
 \vero{For completeness, we recall definitions on max-stable, inverse max-stable and max-mixture processes.}
\subsection{Max-stable model}
Max-stable processes are the extension of the multivariate extreme value theory to the infinite dimensional setting \cite{buhl2016Semiparametric}.  If there exist two sequences of continuous functions 
$(a_n(\cdot)>0,n \in\mathbb{N})$ and $(b_n(\cdot)\in\bR,n \in\mathbb{N})$ such that for all $n\in\mathbb{N}$ and $n$ i.i.d.  $X_1,\ldots,X_n$ and $X$ a process, such that
\begin{equation}\label{th-max}
\max_{i=1,...,n}\frac{X_i-b_n}{a_n} \stackrel{d}\to X,\quad n\to\infty,
\end{equation} 
then $\X$ is a max-stable process \cite{de2006spatial}. When for all $n\in\bN$, $a_n=1$ and $b_n=0$, the margin distribution of the process $X$ is unit Fréchet, that is for any $s\in \ss$ and $x>0$,
$$
F(x):=P(X(s)\leq x)= \exp[-1/x]\/,
$$ 
\vero{in that case, we say that $X$ is a {\em simple max-stable porcess}.}\\
\vero{In \cite{de1984spectral} it is} proved that a max-stable process $X$ can be constructed by using a random process and a Poisson process. This representation is named the \textbf{spectral 
representation}. Let $X$ be a max-stable process on $\ss$. Then there exists $\{\xi_i,i\geq1\}$ i.i.d Poisson point  process on $(0,\infty)$, with intensity $d\xi/\xi^2$ and a sequence 
$\{W_i,i\geq1\}$ of i.i.d. copies of a positive  process $W=(W(s),s\in\ss)$, such that $\bE[W(s)]=1$ for all $s\in\ss$ such that
\begin{equation}\label{max2BB}
X=^d\max_{i\geq1}\xi_iW_i \/.
\end{equation}
\vero{The c.d.f. of a simple max-stable process satisfies: (see~\cite{beirlant2006statistics}, Section 8.2.2.)}
\begin{equation}\label{max44}
\bP\big(X(s_1)\leq x_1,...,X(s_k)\leq x_k\big)=\exp\{-V(x_1,...,x_k)\},
\end{equation}
where the function 
\begin{equation}\label{max45}
V(x_1,...,x_k)=\bE\bigg[\max_{\ell=1,...,k}\bigg(\frac{W(s_\ell)}{x_\ell}\bigg)\bigg].
\end{equation}
is homogenous of order $-1$ and is named \textit{the exponent measure}. 
One of the interests of the exponent measure is its interpretation in terms of dependence. In fact, the homogeneity of the exponent measure $V$ implies
 \begin{equation}\label{max47}
\max\{1/x_1,...,1/x_k\}\leq V(x_1,...,x_k)\leq \{1/x_1+...+1/x_k\}. 
\end{equation}
In Inequalities~(\ref{max47}), the lower (resp. upper) bound corresponds to complete dependence (resp. independence). \\

For simple max-stable process, \textbf{ the extremal coefficient function $\Theta$}, for any pairs of sites $(s,s+h)\in\ss^2$ is the function $\Theta$ defined on $\ss$ (or in $\bR^+$ in isotropic case) with values in $[1,2]$ by 
 \begin{equation}\label{max48}
\bP\big(X(s)\leq x,X(s+h)\leq x\big)=\exp\{-\Theta(h)/x\},\, x>0
\end{equation}
\vero{We have}
 \begin{equation}\label{max50}
\Theta(h)=\bE\big[\max\{W(s),W(s+h)\}\big]=V(1,1)\in[1,2].
\end{equation}
If for any $h\in\ss$, $\Theta(h)=1$ (resp. $\Theta(h)=2$), then we have complete dependence (resp. complete independence). The case  $1<\Theta(h)<2$, for all $h\in\ss$ corresponds to asymptotic 
dependence. Remark that, \vero{for a} simple max-stable process, the coefficients $\Theta$ and $\theta_F$ coincide.

Furthermore, it is easy to see the relationship between $\Theta$ and $\chi$; see~\cite{wadsworth2012dependence} for any $h\in\ss$ 
\begin{equation}\label{relationthetachi}
\Theta(h)=2-\chi(h).
\end{equation}

In the max-stable case,  \cite{cooley2006variograms} gives the relation for all $h\in\ss$,
\begin{equation}\label{extmad}
\Theta(h)=\frac{1+2\nu^F(h)}{1-2\nu^F(h)},
\end{equation} 

which appears to be helpful to estimate the extremal coefficient $\Theta$ \vero{from the empirical madogram}. The max-stable process $X$ with pairwise distribution function is given by the following 
equation, for all $(s,s+h)\in\ss^2$,
\begin{equation}\label{max52}
\bP\big(X(s)\leq x_1,X(s+h)\leq x_2\big)=\exp\{-V_h(x_1,x_2)\},
\end{equation} 
\\
We provide \vero{below} three examples of well-known max-stable models represented by different exponent measures $V$.\\

{\textbf{Smith Model (Gaussian extreme value model)}} \cite{smith1990max} with unit Fr\'echet margin and exponent measure 
\begin{equation}\label{max53}
V_h(x_1,x_2)=\frac{1}{x_1}\Phi\bigg(\frac{\tau(h)}{2}+\frac{1}{\tau(h)}\log\frac{x_2}{x_1}\bigg)+\frac{1}{x_2}\Phi\bigg(\frac{\tau(h)}{2}+\frac{1}{\tau(h)}\log\frac{x_1}{x_2}\bigg);
\end{equation} 
$\tau(h)=\sqrt{h^T\Sigma^{-1}h}$ and $\Phi(\cdot)$ the standard normal cumulative distribution function. For isotropic case $\tau(h)=\sqrt{\frac{1}{\sigma}\|h\|^2}$.\\
 The pairwise extremal coefficient equals
$$
\Theta(h)=2\Phi\bigg(\frac{{\tau(h}}{{2}}\bigg).
$$
\textbf{Brown-Resnik Model} \cite{kabluchko2009stationary}  with unit Fr\'echet margin and exponent measure
$$
V_h(x_1,x_2)=\frac{1}{x_1}\Phi\bigg(\frac{\sqrt{2\gamma(h)}}{2}+\frac{1}{\sqrt{2\gamma(h)}}\log\frac{x_2}{x_1}\bigg)+\frac{1}{x_2}\Phi\bigg(\frac{\sqrt{2\gamma(h)}}{2}+\frac{1}{\sqrt{2\gamma(h)}}\log\frac{x_1}{x_2}\bigg);
$$
$2\gamma(h)$ is a variogram and $\Phi(\cdot)$ the standard normal cumulative distribution function. 
The pairwise extremal coefficient given by
$$\Theta(h)=2\Phi\bigg(\frac{\sqrt{2\gamma(h)}}{2}\bigg).
$$
{ \textbf{Truncated extremal Gaussian Model (TEG)}} \cite{schlather2002models} with unit Fr\'echet margin and exponent measure

\begin{equation}\label{max59}
V_h(x_1,x_2)=\bigg(\frac{1}{x_1}+\frac{1}{x_2}\bigg)\bigg[1-\frac{\alpha(h)}{2}\bigg(1-\sqrt{1-2(\rho(h)+1)\frac{x_1x_2}{(x_1+x_2)^2}}\bigg)\bigg].
\end{equation} 
The extremal coefficient is given by
 \begin{equation}\label{max60}
\Theta(h)=2-\alpha(h)\left\{1-\bigg(\frac{1-\rho(h)}{2}\bigg)^{1/2}\right\}
\end{equation} 
where $\alpha(h)=\bE\{|\mathcal{B}\cap(h+\mathcal{B})|\}/\bE[|\mathcal{B}|]$ \vero{and} $\mathcal{B}$ is a random set which can consider  it a disk with fixed radius $r$. \vero{In such a 
case,} $\alpha(h)=\{1-h/2r\}_+$ (see \cite{davison2013geostatistics}) \vero{and} $\chi(h)=0,\forall h\geq2r$. 

\subsection{Inverse Max-stable processes}
\vero{Max-stable processes are either asymptotically Dependent (AS) or independent. This means that they are not useful to model non trivial AI processes. } In,\cite{wadsworth2012dependence} a class 
of asymptotically  independent processes \vero{is}  obtained by inverting max-stable processes. These processes are called  \textbf{inverse max-stable processes}; \vero{their survivor 
function satisfy Equation}~(\ref{ind1}).  Let $X':=\{X'(s),s\in\ss\}$ be a \vero{simple} max-stable process \vero{with exponent measure $V$.}
Let  $g:(0,\infty)\mapsto(0,\infty)$ be  defined by 
$g(x)=-1/\log\{1-e^{-1/x}\}$ \vero{and} $X(s)=g(X'(s))$. Then, $\X$ is an  asymptotic independent spatial process with unit Fréchet margin. The $d$-dimensional joint survivor function \vero{satisfies}
\begin{equation}\label{ind5}
\begin{split}
\bP\big(X(s_1)>x_1,...,X(s_d)>x_d\big)=&\exp\left\{-V\big(g(x_1),...,g(x_d)\big)\right\}\/.
\end{split}
\end{equation}    
The tail dependent coefficient is given by $\eta(h)=1/\Theta(h)$, where $\Theta(h)$ is the extremal coefficient of the max-stable process $X'$.  Moreover, we have $\widetilde{\chi}(h)=2/\Theta(h)-1$. 
\vero{With a slight abuse of notations, we shall say that $V$ is the exponent measure of $X$.}

\subsection{Max-mixture model}\label{max-mixture}
In spatial contexts, specifically in an environmental domain, many scenarios of dependence could arise and AD and AI might cohabite. The work by \cite{wadsworth2012dependence} provides a flexible model called max-mixture.\\
Let $\X$  be a \vero{simple} max-stable process with extremal coefficient $\Theta(h)$ and bivariate distribution function $F_X$, and  let $\Y$ be an \vero{inverse max-stable process} whose 
 tail dependence \vero{coefficient is} $\eta(h)$. \vero{Its bivariate distribution function is denoted}  $F_Y$. Assume that $X$ and 
$Y$ are independent. Let  $a\in[0,1]$ and \vero{define}
\begin{equation}\label{max1}
Z(s)=\max\{aX(s),(1-a)Y(s)\}, \quad s\in\ss,
\end{equation} 
then $Z$ unit  Fréchet marginals and \vero{its pairwise survivor function satisfies}
\begin{equation}\label{max3}
\bP\big(Z(s)>z,Z(t)>z\big)\sim\frac{a\{2-\Theta(h)\}}{z}+\frac{(1-a)^{1/\eta(h)}}{z^{1/\eta{(h)}}}+O(z^{-2}),\quad z\to\infty.
\end{equation}  
\vero{The process $(Z(s))_{s\in\ss}$ is called a max-mixture process.} Assume there exists  finite $h^*=inf\{h:\Theta(h)\neq 0\}$; then,
\begin{equation}\label{co1}
\chi(h)=a(2-\Theta(h))
\end{equation}
and 
\begin{equation}\label{co2}
\overline{\chi}(h)=\mathds{1}_{[h^*< h]}(h)+(2\eta(h)-1)\mathds{1}_{[h*\geq h)}.
\end{equation}
\begin{remark}
If there exists  finite $h^*=inf\{h:\Theta(h)\neq 0\}$, then $Z$ is asymptotically dependent up to distance $h^*$ and asymptotically independent for larger distances.\\
\vero{Of course, if $a=0$ then $Z$ is AI and if $a=1$ then $Z$ is a simple max-stable process.}
\end{remark}
\vero{In}  \cite{bacro2016flexible} \vero{max-mixture processes are studied. The authors emphasize the fact that these models allow asymptotic dependence and independence  to be present at a short 
and intermediate  distances. Furthermore, the process may be independent at  long distances (using e.g. TEG processes).} 
 \section{$F$-madogram for max-mixture spatial process}\label{Se:3}
In extreme value theory and therefore for spatial extremes, one of the main concerns is to find a dependence measure that can quantify the dependences between locations. 
The $\chi$ and $\overline{\chi}$ dependence measures are designed to quantify asymptotic dependence and asymptotic independence respectively \vero{(see equations (\ref{co1}) and 
(\ref{co2}))}.  Max-mixture processes have been introduced in order to 
provide both behaviors.  We are then faced with the question of finding an adapted tool which would give information on more than one dependence structure.\\ 

In \cite{cooley2006variograms},  the $F$-madogram has been introduced for max-stable processes. There exists several definitions of madograms. For example, in \cite{naveau2009modelling}, the
$\lambda$-madogram is considered  in order to take into account the dependence information from the exponent measure $V_h(u,v)$ when $u\neq v$. This  $\lambda$-madogram has been extended in 
\cite{fonseca2011generalized}  to evaluate the dependence between two observations located in two disjoint regions  in $\bR^2$. \cite{guillou2014madogram} adopted a $F$-madogram suitable for 
asymptotic independence instead of asymptotic dependence only. Finally, \cite{bacro2010testing} used F-madogram as a test statistic for  asymptotic independence bivariate maxima. \\
\\
Below, we calculate $\nu^F(h)$ for a max-mixture process. It appears that contrary to $\chi$ and $\overline{\chi}$, it combines the parameters coming from the AD and the AI parts.\\
 \begin{Pro}\label{madogram}
Let $Z$ be a max-mixture process, with mixing coefficient $a\in[0\/,1]$. Let $X$ be its max-stable part with extremal coefficient $\Theta(h)$. Let $Y$ be its  inverse max-stable part with tail 
dependence coefficient $\eta(h)$. Then,  the $F$-madogram  of $Z$ is \vero{given by}
\begin{multline} \label{mado:1} 
\nu^F(h)=\\
\frac{a(\Theta(h)-1)}{a(\Theta(h)-1)+2}-\frac{a\Theta(h)-1}{2a\Theta(h)+2}- \frac{1/\eta(h)}{a\Theta(h)+(1-a)/\eta(h)+1}\beta\bigg(\frac{a\Theta(h)+1}{(1-a)},1/\eta(h)\bigg)\/,
\end{multline}
where $\beta$ is beta function. 

 \end{Pro}
\begin{proof}
We have
\begin{equation} \label{defmado}
\nu^F(h)=\frac12\bE |F(Z(s))-F(Z(s+h))|. 
\end{equation}
The equality $|a-b|/2=\max(a,b)-(a+b)/2$ leads to 
\begin{equation} \label{defmado1}
\nu^F(h)=\bE \big[\max\big(F(Z(s)),F(Z(s+h))\big)\big]-\bE\big[F(Z(s))\big]. 
\end{equation}
Let $M(h)=\max\big(F(Z(s)),F(Z(s+h))\big)$, we have:
\begin{equation}\label{prob11}
\bP\big(M(h)\leq u\big)=\bP\big(Z(s)\leq F^{-1}(u), Z(s+h)\leq F^{-1}(u)\big).
\end{equation}
From \vero{definition} of the max-mixture spatial process $Z$, \vero{we have}
\begin{equation}
\bP\big(Z(s)\leq z_1,Z(s+h)\leq z_2\big)=e^{-aV_X^h(z_1,z_2)}\bigg[e^{\frac {-(1-a)} {z_1}}+e^{\frac {-(1-a)} {z_2}}-1+e^{-V_Y^h(g(z_1),g(z_2))}\bigg],
\end{equation} 
 where $V_X$ (resp.$V_Y$) corresponding to the exponent measures of $X$ (resp. $Y$) and $g(z)=-1/\log(1-e^{\frac{-(1-a)}{z}})$. That leads to
\begin{equation*}\label{distM}
\bP\big(M(h)\leq u\big)= u^{a\Theta(h)}\big[2u^{(1-a)}-1+\big(1-u^{(1-a)}\big)^{1/\eta(h)}\big],\quad u\in[0,1]\/.
\end{equation*}
We deduce that 
\begin{equation}\label{drivmad}
\begin{split}
\bE[M(h)]=&\int_0^1uf_M(h)(u)\mathrm{d}u\\
=&\frac{2a(\Theta(h)-1)+2}{a(\Theta(h)-1)+2}-\frac{a\Theta(h)}{a\Theta(h)+1}-\frac{\beta\bigg(\frac{a\Theta(h)+1}{(1-a)},1/\eta(h)\bigg)}{\eta(h)(1-a)\bigg[\frac{a\Theta(h)+1}{(1-a)}+(1/\eta(h))\bigg]}.
\end{split} 
\end{equation}
where $f_{M(h)}$ is the density of $M(h)$. Recall that $\E(F(Z(s)))=\frac12$ because $F(Z(s))\sim \UU([0\/,1])$ and return to equation (\ref{defmado1}) to get  equation (\ref{mado:1}).
\end{proof}

In the particular cases where $a=1$ or $a=0$, Proposition \ref{madogram} reduces to known results for max-stable processes (see \cite{cooley2006variograms}) and inverse max-stable processes (see 
\cite{guillou2014madogram}). \vero{That is}, the $F$-madogram for a max-stable spatial process  is given by
 \begin{equation}\label{mdoXX}
2 \nu^F(h)=\frac{\Theta(h)-1}{\Theta(h)+1}. 
 \end{equation}
and the  $F$-madogram of an asymptotically independent spatial process is given by 
 \begin{equation}\label{asymado}
2 \nu^F(h)=\frac{1-\eta(h)}{1+\eta(h)}
 \end{equation}
 In order to have a comprehensive view of the behavior of $\nu^F(h)$, we have plotted in Figure \ref{madoFig:1} below $h\leadsto \nu^F(h)$. We have considered two max-mixture models MM1 and 
MM2 described as \vero{follows}:
\begin{itemize}
 \item \textbf{MM1} max-mixture between a TEG max-stable process $X$ and an inverse Smith max-stable process $Y$;  
 \item \textbf{MM2} max-mixture between  $X$ \vero{as} in MM1  and \vero{an} inverse TEG max-stable process $Y$. 
\end{itemize}
 
In this Figure and for the two models MM1 and MM2, $\nu^F(h)$ has two sill  one corresponding to  $X$ and the second corresponding to $Y$. This is completely in accordance with the nested variogram concept as presented in \cite{wackernagel1998multivariate}. In data analysis, these two levels of the sill gives the researcher a hint about whether there is more than one spatial dependence structure in the data. \\
 \vero{Since the} $F$-madogram expresses with all the model parameters \vero{it should be useful} for the parameter estimation.  
 \begin{figure}[h]\label{madog:1}
\centering
\fboxrule=0pt 
\fbox{\includegraphics[width=12cm,height=7cm]{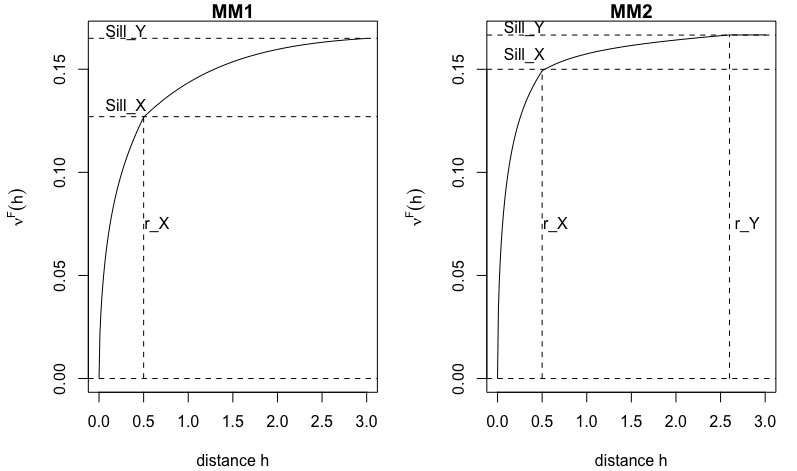}}\
\caption{ $h\leadsto \nu^F(h)$ for the max-mixture processes models MM1 and MM2. In MM1, $X$ has correlation function $\rho(h)=\exp(-h/\theta_X)$, $r_X=0.25$ and fixed radius $\theta_X=0.2$ and $Y$ 
has $\Sigma=\sigma_Y \bf I_d$, $\sigma_Y=0.6$. In MM2, $X$ \vero{has} the same setting \vero{as} in MM1 and $Y$ has $\rho_Y(h)=\exp(-h/\theta_Y)$, $\theta_Y=0.8$ and  fixed radius $r_Y=1.35$.  For 
the two models, we set $a=0.5$. \label{madoFig:1}}
\end{figure}

\section{Model inference }\label{Se:4}
This section is devoted to the parametric inference for max-mixture processes. We begin with the presentation of the maximum composite likelihood estimation, then we present the least squares 
madogram \vero{estimation}.
\subsection{Parametric Estimation using Composite Likelihood } \ \\
Consider $(Z^k(s_1)\/,\ldots\/, Z^k(s_D))$, $k=1\/,\ldots\/, N$, be $N$ independent copies of a spatial process $(Z(s))_{s\in\ss}$, observed at $D$ locations $s_1\/,\ldots\/,s_D$.  Composite 
likelihood inference is  \vero{an} appropriate approach in  estimating  the parameter models of a spatial process $X$ (\cite{lindsay1888composite,varin2011overview}). \vero{Asymptotic} properties of 
this estimator has been proved in \cite{davis2011comments}. This approach has been applied successfully  to spatial max-stable processes  by \cite{davison2012geostatistics} and 
\cite{padoan2010likelihood} and is also used to identify the parameters of data exceedances over a large threshold, for example, \cite{bacro2014estimation} and \cite{thibaud2013threshold}. 
\vero{In this article, we focus on max-mixture models. Composite likelihood inference for max-mixture processes has been studied in \cite{bacro2016flexible} and \cite{wadsworth2012dependence}. 
We will compare our madogram based estimation to the composite likelihood  estimation described in \cite{bacro2016flexible}. }\\
\\
If the pairwise density of $Z$ can be computed and its  parameter $\psi$ is identifiable, then it is possible to estimate $\psi$ by maximizing the pairwise weighted log likelihood. For simplicity, we denote $Z_i^k$ for $Z^k(s_i)$. Let  
\begin{equation*}
\widehat{\psi}_L=\max_{\psi}\mathcal{P}(\psi),
\end{equation*}
where
\begin{equation}\label{likli:1}
\mathcal{P}(\psi)=\sum_{k=1}^N\sum_{i=1}^{D-1}\sum_{j>i}^{D}w_{ij}\log\mathscr{L}(Z_i^k,Z_j^k;\psi)=:\sum_{k=1}^N \mathcal{P}_k(\psi)\/.
\end{equation}
 where $\mathscr{L}$ is the likelihood of the pair $(Z_i^k,Z_j^k)$ and $w_{i,j}\geq0$ is the weight that specifies the contribution for each pair.  In \cite{padoan2010likelihood}, it is suggested to 
 \\
In \cite{coles2001introduction}, it is suggested to consider a  censor approach of the  likelihood, taking into account a threshold. This approach has been adopted in this study.  Let $G(\cdot,\cdot)$ be a pairwise distribution function  and consider  the thresholds $u_1$ and $u_2$; the likelihood contribution is 
\[ \mathscr{L}(z_1,z_2;\psi) =
  \begin{cases}
   \partial^2_{12} G(z_1,z_2;\psi)      &\mbox{if} \ z_1>u_1,z_2>u_2,\\
                         G(z_1,z_2;\psi)     &\mbox{if} \ z_1\leq u_1,z_2 \leq u_2,
  \end{cases}
\]
In \cite{wadsworth2012dependence}, the censored likelihood is used in order to improve the estimation of the parameters 
related to asymptotic independence.  This censored approach was also applied by  \cite{bacro2016flexible} for the estimation of parameters of max-mixture processes. In \vero{that} paper, the 
replications $Z^1\/,\ldots\/, Z^N$ of $Z$ are assumed to be  $\alpha$-mixing rather than independent.  We denote generically by $\psi$ the parameters of the model. In \cite{bacro2016flexible}, it is 
proved, under some smoothness assumptions on the composite likelihood, that the composite maximum likelihood estimator $\widehat{\psi}_L$ for max-mixture processes is asymptotically normal as $N$ goes 
to infinity with asymptotic variance
 \begin{equation*}
\mathscr{G}(\psi)=\mathcal{J}(\psi)[\mathcal{K}(\psi)]^{-1}\mathcal{J}(\psi),
 \end{equation*} 
where $\mathcal{J}(\psi)=\bE[-\nabla^2\mathcal{P}(\psi)]$, $\mathcal{K}(\psi)=\mathrm{var}(\nabla\mathcal{P}(\psi))$.  The matrix $\mathscr{G}(\psi)$ is called the Godambe information matrix (see 
\cite{bacro2016flexible} and Theorem 3.4.7 in \cite{guyon1995random}). \\
An estimator  $\widehat{\mathcal{J}}$ of $\mathcal{J}(\psi)$ is obtained from the Hessian matrix computed in the  optimization algorithm. The  variability matrix   ${\mathcal{K}}(\psi)$ has to be 
estimated too.  In our context, we have independent replications of $Z$ and $N$ is large compared with respect to the dimension of $\psi$. Then, we can use the outer product of the estimation of 
$\widehat{\psi}$. Let
$$\widehat{\mathcal{K}}(\psi)=N^{-1}\sum_{k=1}^{N}\nabla\mathcal{P}_k(\widehat{\psi})\nabla\mathcal{P}_{k}(\widehat{\psi})'$$ 
 or by  Monte Carlo simulation with explicit formula of $\mathcal{P}_k(\psi)$ (see section 5. in \cite{varin2011overview}). In the case of samples of $Z$ satisfying the $\alpha$-mixing property, the estimation of  $\mathcal{K}(\psi)$ can be done using a subsampling technique introduced by \cite{carlstein1986use}; this was used in  \cite{bacro2016flexible}. Finally, model selection can be done by using the composite likelihood  information criterion \cite{varin2005note}:
 \begin{equation*}
  \mathrm{CLIC}=-2\bigg[\mathcal{P}(\widehat{\psi})-tr(\widehat{\mathcal{J}}^{-1}\widehat{\mathcal{K}})\bigg].
 \end{equation*}
Considering several max-stable models, the one that has the smallest CLIC will be chosen. In \cite{thibaud2013threshold}, the criterion $\mathrm{CLIC}^*=(D-1)^{-1}\mathrm{CLIC}$ is proposed. It is close to Akaike information criterion (AIC).
\subsection{Semi-parametric estimation using NLS of F-madogram}\label{estimator:1}  
 In this section, we shall define the non-linear least square estimation  procedure of the parameters set $\psi$ corresponding to the max-mixture model $Z$ using the $F$-madogram. This procedure can be considered  as an alternative method to the composite likelihood method.\\
Consider \vero{$Z^k$, $k=1\/,\ldots\/, N$  copies} of an  isotropic  max-mixture process $Z$ with unit Fréchet marginal laws ($F$ denotes the distribution function of a unit Fréchet law). It may be 
independent copies for example, if the data is recorded yearly (see \cite{naveau2009modelling}) or we shall consider that \vero{$(Z^k)_{k=1\/,\ldots\/, N}$ satisfies a} $\alpha$-mixing property 
(\cite{bacro2016flexible}). Let $\mathcal{H}$ be a finite subset of $\ss$, $J(x,y)=\frac{1}{2}|x-y|$ and $Y_{h,k}=J\big(F(Z^k(s)),F(Z^k(s+h))\big)$, $k=1,..,N$ and $h\in\mathcal{H}$. Therefore, for 
$k=1\/,\ldots\/, N$, the vectors $(Y_{h\/,k})_{h\in\mathcal{H}}$ have the same law and are considered either independent or $\alpha$-mixing (in $k$). The main motivation for using the F-madogram in 
estimation is  that it contains the dependence structure information for a fixed $h$ of $Y_{h,k}$ (see Section 3.2 in \cite{bacro2010testing}). \\
In what follows, we make the assumption that the vectors $(Y_{h\/,k})_{h\in\mathcal{H}}$ are i.i.d. Note that  from the  definition of the  $F$-madogram, we have $\bE[Y_{h,k}]=\nu^F(h,\psi)$ where 
$\nu^F(h,\psi)$ is the  $F$-madogram  of $Z$ with parameters $\psi$ defined in (\ref{mado:1}). If $Z$ has an unknown true parameter $\psi_0$ on a compact set $\Psi\subset \R^d$, we rewrite
\begin{equation}\label{reg:1}
Y_{h,k}=\nu^F(h,\psi_0)+\varepsilon_{h,k}.
\end{equation}  
The vectors $(\varepsilon_{h,k})_{h\in\mathcal{H}}$ are i.i.d errors with  $\bE[\varepsilon_{h,k}]=0$ and $\mbox{Var}(\varepsilon_{h,k})=\sigma_h^2>0$ is finite and unknown. \\
Let 
\begin{equation}\label{reg:1B}
\mathcal{L}(\psi)=\sum_{h\in\mathcal{H}}\frac{1}{N}\sum_{k=1,...,N}\big(Y_{h,k}-\nu^F(h,\psi)\big)^2
\end{equation}
Any vector $\widehat{\psi}_M$ in $\Psi$ which minimizes $\mathcal{L}(\psi)$ will be called a least square estimate of $\psi_0$:
\begin{equation}\label{reg:2}
\widehat{\psi}_M\in \argmin_{\psi\in\Psi} \mathcal{L}(\psi) \/.
\end{equation}

\begin{Th}\label{reg:4}
Assume that $\Psi\subset\R^d$ is compact and that $\psi\mapsto \nu^F(h\/,\psi)$ is continuous for all $h\in\mathcal{H}$. We assume that the vectors $(Y_{h\/,k})_{h\in\mathcal{H}}$ are i.i.d. Let 
\vero{$(\widehat{\psi}_M^N)_{N\in\N}$ be least square estimators of $\psi_0$.  Then, any limit point (as $N$ goes to infinity) $\psi$ of $(\widehat{\psi}_M^N)_{N\in\N}$} satisfies $\nu(h\/,\psi) = 
\nu(h\/,\psi_0)$ for all $h\in\mathcal{H}$. 
\end{Th}
\begin{remark} Of course, if $\psi\leadsto (\nu(h\/,\psi))_{h\in\mathcal{H}}$ is injective, then \vero{Theorem \ref{reg:4}} implies that the least square estimation is consistent, i.e. 
$\widehat\psi_M^N\rightarrow \psi_0$ a.s. as $T$ goes to infinity. In the examples considered below, it seems that the injectivity is satisfied provided $|\mathcal{H}| \geq d$, but we were unable to 
prove it.
\end{remark}

\begin{proof}
We follow the proof of Theorem II.5.1 in \cite{antoniadis1992regression}.
From  (\ref{reg:1}), we have, for   all $\psi\in\Psi$ 
\begin{equation*}
\begin{split}
\mathcal{L}(\psi)=&\sum_{h\in\mathcal{H}}\frac{1}{N}\sum_{k=1,...,N}\big(\nu^F(h,\psi_0)+\varepsilon_{h,k}-\nu^F(h,\psi)\big)^2\\
=&\sum_{h\in\mathcal{H}}\big(\nu^F(h,\psi_0)-\nu^F(h,\psi)\big)^2\\
+&\frac{2}{N}\sum_{h\in\mathcal{H}}\big(\nu^F(h,\psi_0)-\nu^F(h,\psi)\big)\left(\sum_{k=1,...,N}\varepsilon_{h,k}\right)+\sum_{h\in\mathcal{H}}\frac{1}{N}\sum_{k=1,...,N}\varepsilon_{h,k}^2.
\end{split}
\end{equation*} 
From the law of large numbers, we have 
\begin{equation*}
\frac{1}{N}\sum_{h\in\mathcal{H}}\sum_{k=1,...,N}\varepsilon_{h,k}^2\to\sum_{h\in\mathcal{H}}\sigma_h^2\quad \mathrm{a.s.}\quad\mathrm{as} \quad N\to\infty
\end{equation*}
and for any $h\in\mathcal{H}$,
$$\frac{1}{N}\sum_{k=1,...,N}\varepsilon_{h,k}\to 0 \ \mbox{a.s.}$$

Therefore, 
\begin{equation*}
\mathcal{L}(\psi)\to  \sum_{h\in\mathcal{H}}\sigma_h^2 + \sum_{h\in\mathcal{H}}\big(\nu^F(h,\psi_0)-\nu^F(h,\psi)\big)^2\quad \mathrm{a.s.}\quad\mathrm{as} \quad N\to\infty.
\end{equation*}
Take a sequence $(\widehat{\psi}_M^N)_{N\in\N}$ of least square estimators, taking if necessary a subsequence, we may assume that it converges to some $\psi^*\in\Psi$. Using the continuity of $\psi 
\leadsto \nu^F(h\/,\psi)$, we have
$$\mathcal{L}(\widehat{\psi}_M^N) \rightarrow \sum_{h\in\mathcal{H}}\sigma_h^2 + \sum_{h\in\mathcal{H}}\big(\nu^F(h,\psi_0)-\nu^F(h,\psi^*)\big)^2\quad \mathrm{a.s.}\quad\mathrm{as} \quad 
N\to\infty.$$
Since $\widehat{\psi}_M^N$ is a least square estimator, $\mathcal{L}(\widehat{\psi}_M^N)\leq \mathcal{L}(\psi_0)\rightarrow \displaystyle \sum_{h\in\mathcal{H}}\sigma_h^2$. It follows that
$$\sum_{h\in\mathcal{H}}\big(\nu^F(h,\psi_0)-\nu^F(h,\psi^*)\big)^2 =0$$
and thus $\nu(h\/,\psi^*) = \nu(h\/,\psi_0)$ for all $h\in\mathcal{H}$. 
 \end{proof}
The asymptotic normality of the least square estimators should also be obtained by following, e.g., \cite{buhl2016Semiparametric} and using the asymptotic normality of the $F$-madogram obtained in  \cite{cooley2006variograms}. Nevertheless, the calculation of the asymptotic variance will require to calculate the covariances between $\nu^F(h_1\/,\psi)$ and $\nu^F(h_2\/,\psi)$, which is not straightforward.
\section{Simulation study}\label{Se:5}
This section is devoted to some simulations in order to evaluate the performance of the least square estimator and to compare it with the maximum composite likelihood estimator. Recall that 
$\widehat{\psi}_M$ denotes the least square estimator of the parameter vector $\psi$ and $\widehat{\psi}_L$ denotes the composite likelihood estimator.

\subsection{Outline the \vero{simulation} experiment }

In order to evaluate the performance  of the non-linear least square estimator $\widehat{\psi}_M$ as defined in (\ref{mado:1}), we have generated data from the model MM1 above.  
The \vero{least square madogram} estimator $\widehat{\psi}_M$ has been  compared with true one $\psi_0$  and also with parameters estimated by \vero{maximum} composite likelihood  $\widehat{\psi}_L$ 
proposed in \cite{bacro2016flexible,wadsworth2012dependence}, \vero{on the same simulated} data. We considered $50$ sites randomly and uniformly distributed in the square $\cA=[0,1]^2$. \\
We have generated  $N=1.000$ i.i.d observations for each site \vero{and replicated this experiment}  $J=100$  times. We have considered several  mixing parameters: 
$a:=\{0,0.25,0.5,0.75,1\}$. For the  composite likelihood estimator $\widehat{\psi}_L$, we used the censored procedure with \vero{the $0.9$ empirical quantile of data at each site as threshold $u$}. 
The fitting of  $\widehat{\psi}_L$ was done using the code which was used in  \cite{bacro2016flexible} with some \vero{appropriated} modifications.
\subsection{Results on the parameters estimate}
\vero{In Figure \ref{EST:1}, we represented the boxplots of the errors}, that is $(\widehat{\psi}_M-\psi_0)$ and $(\widehat{\psi}_L-\psi_0)$ for model 
MM1. Generally, the estimators above worked well, although the variability in some estimates were relatively large, especially for the asymptotic  independence parameters. It also shows some bias in 
the estimation  of  \vero{the} asymptotic independence model parameters. \\
\\
It is well known that asymptotic independence is  difficult to estimate (see \cite{davison2013geostatistics}). Therefore, the estimation accuracy of the parameters is very sensitive. On one other 
hand, the fitting of $\alpha(h)$ which appears in TEG models in (\ref{max59}), is delicate and might \vero{lead to quite different estimates} with different data 
\cite{davison2012geostatistics}.
\begin{figure}[!htb]
\centering
\fboxrule=0pt 
\fbox{\includegraphics[width=12cm,height=6cm]{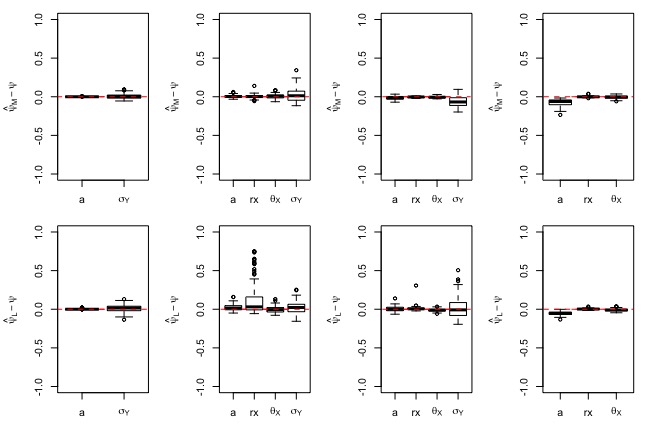}}\
\caption{Boxplots display $(\widehat{\psi}-\psi)$ of estimated parameters vector $\widehat{\psi}=(\widehat{a},\widehat{r}_X,\widehat{\theta}_X,\widehat{\sigma}_Y)^T$ for the MM1 model by the two 
estimators $\widehat{\psi}_M$ and $\widehat{\psi}_L$. The figures in the first row and from left to right concern the estimator $\widehat{\psi}_M$ for $a\in\{0,\vero{0.25},0.75,1\}$, the second row 
concerns 
$\widehat{\psi}_L$. We have set, $r_{X}=0.25$, $\theta_{X}=0.20$ and $\sigma_{Y}=0.6$ over a square $\cA=[0,1]^2$.  \label{EST:1}}
\bigskip
\fbox{\includegraphics[width=12cm,height=6cm]{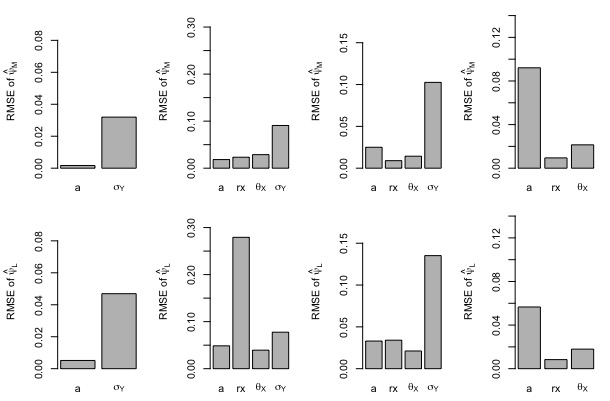}}\
\caption{Barplots display the RMSE of $\widehat{\psi}$ for each estimated parameters $\widehat{\psi}=(\widehat{a},\widehat{r}_X,\widehat{\theta}_X,\widehat{\sigma}_Y)^T$ for MM1 and the corresponding  
 two estimators $\widehat{\psi}_M$ and $\widehat{\psi}_L$. The bars in the first row and from left to right represent the RMSE of the estimator $\widehat{\psi}_M$  when $a:=\{0,\vero{0.25},0.75,1\}$, 
respectively and the same for the second row for $\widehat{\psi}_L$. We set $r_{X}=0.25$, $\theta_{X}=0.20$ and $\sigma_{Y}=0.6$ over a square $\cA=[0,1]^2$.  \label{EST:2}}
\end{figure}
\vero{Another comparison indicator is } the root mean square error (RMSE) (\cite{zheng2015assessing,zheng2014modeling}): \vero{let $\widehat{\psi_j}$ denote the $j$th estimation } (either least 
square or composite likelihood estimation),
 \begin{equation}
 \mathrm{RMSE}=\bigg[J^{-1}\sum_{j=1}^J(\widehat{\psi_j}-\psi)^2\bigg]^{1/2},
 \end{equation}  
The barplot in Figure \ref{EST:2}  displays the  RMSE for each parameter of MM1 model. We see on these barplots that when $a$ is close to $0$ ($a=0;0.25$), the estimator $\widehat{\psi}_M$ 
over-performs  the estimator $\widehat{\psi}_L$ and vice versa when $a\in\{0.75,1\}$. For $a=0.5$ the performance of the two estimators seem  equivalent.
\subsection{Asymptotic normality}
\vero{We may figure out whether the least square estimator $\widehat{\psi}_M$ is asymptotically normal through the graphe of the errors $\widehat{\psi}_L-\psi_0$.} In Figure \ref{DEEN:1}, the graphs 
represent the 
distributions of  the errors  of each \vero{estimated} parameters of MM1 model for $a:=\{0,0.5,1\}$. We \vero{did}  $J=500$ simulation experiments.  \\
We can see \vero{on  Figure \ref{DEEN:1} that} the densities of the errors of the parameters  \vero{seem close to the shape of centered} normal distribution. 
\begin{figure}[h]
\centering
\fboxrule=0pt 
\includegraphics[width=13cm,height=7cm]{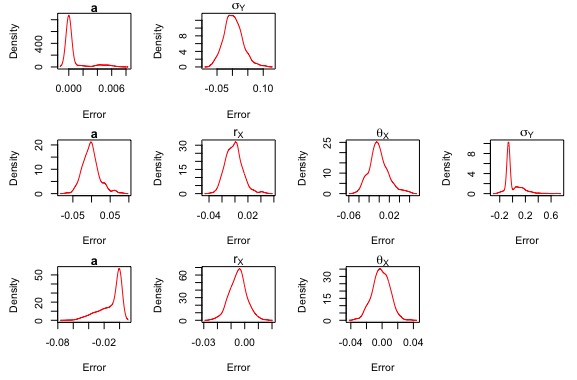}
\caption{Graphs  display the densities  of the errors between $\widehat{\psi}_M$  and $\psi_0$ for each estimated parameters in the set $\psi=(a,r_X,\theta_X,\sigma_Y)^T$ of MM1 model.  The graphs 
from first row and last one represent the densities of the \vero{errors $\widehat{\psi}_M-\psi_0$}  when 
$a:=\{0,0.5,1\}$. We set $r_{X}=0.25$, $\theta_{X}=0.20$ and $\sigma_{Y}=0.6$ over a square $\cA=[0,1]^2$.  \label{DEEN:1}}
\end{figure}

\section{Real data example}\label{Se:6}
In this section, we analyze \vero{a real data and fit} it to some models considered in this study by composite likelihood and LS-madogram procedures.
\subsection{Data \vero{analysis}}
 We analyzed daily rainfall \vero{along the} east coast of Australia. We selected  39 locations randomly from \vero{this} region. \vero{The data is } daily measured in the period \vero{from April to 
September} for 35 years from 1982-2016.  The data \vero{is} available \vero{from the}  Australian Bureau of Meteorology (http://www.bom.gov.au/climate/data/). \\

\vero{In order to} explore the possibility of anisotropy of the spatial dependence, we used the same test \vero{as} in  \cite{bacro2016flexible}. We divided all data set according to directional 
sectors $(-\pi/8,\pi/8]$, $(\pi/8,3\pi/8]$, $(3\pi/8,5\pi/8]$ and $(5\pi/8,7\pi/8]$, where $0$ indicate to north direction. \vero{We use}  the empirical F-madogram 
$\widehat{\nu}^F(h)$. The directional loss smoothing of such empirical measure in the Figure \ref{madoFig:11}. (A), shows no evidence of anisotropy. \\

\vero{In Figure \ref{madoFig:11}. (B), the empirical F-madogram is plotted for the whose data set. It seems that asymptotic  dependence between the locations is} present up to 
distance 500 km and asymptotic independence could be present \vero{at} the remaining distance. \\
\begin{figure}[h]
\centering
\fboxrule=0pt 
\fbox{\includegraphics[width=13cm,height=7cm]{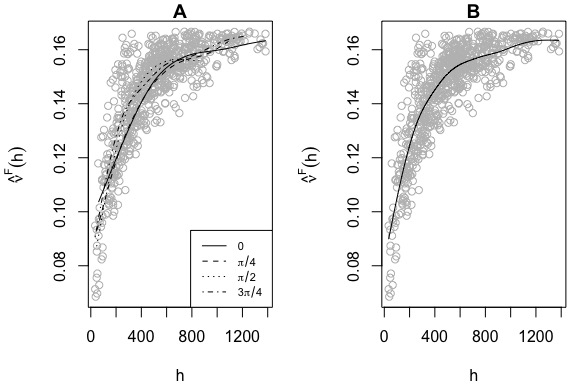}}\
\caption{ Empirical evaluation of $\widehat{\nu}^F(h)$. The Grey circles represented the empirical value \vero{between all pairs}. The lines in (A) represent the smoothed value of the empirical  of 
$\widehat{\nu}^F(h)$ according to directional sectors $(-\pi/8,\pi/8]$, $(\pi/8,3\pi/8]$, $(3\pi/8,5\pi/8]$ and $(5\pi/8,7\pi/8]$. The line in (B) represent the smoothed value of the empirical  of 
$\widehat{\nu}^F(h)$ for all directions. \label{madoFig:11}}
\end{figure}
\subsection{Data fitting}
Our interest in this section is to chose  \vero{a} reasonable model \vero{for} the data. We considered  $7$ models described below, for each model,  the parameters \vero{are} estimated by 
LS-madogram and \vero{maximum} composite likelihood. The selection criteria for LS-madogram estimator $\widehat{\psi}_M$ \vero{is} computed as the following: $$MIC:=\log{ 
\mathcal{L}(\widehat{\psi})}+(2k(k+1)/(T-k-1)),$$
where $k$ is the number of parameters  in \vero{the model and $T$ is the number of the observations, that is: $T=N\times|\mathcal{H}|$, where $|\mathcal{H}|$ is the number of observed pairs}. With 
respect to censored composite likelihood estimators we adopted the  CLIC selection criteria. The two criteria selected the model MM1 as the best model.  \\ 
\vero{We consider the following models}:  
\begin{description}
\item[MM1]  max-mixture between asymptotic dependence process represented by  TEG max-stable process $X$ with exponential  correlation function $\rho(h)=\exp\{-(h/\theta_X)\}, \theta_x>0$ and 
$\mathcal{B}_X$ is a disk of fixed and unknown radius $r_X$. \vero{The} asymptotic independence \vero{is represented by an } inverse Brown-Resnik  max-stable process $Y$ with variogram 
$2\gamma(h)=2\sigma^2{(1-\exp\{-(h/\theta_Y)\})}$, $ \theta_Y, \sigma >0$; $\sigma^2$ is the sill of the variogram. 
\item[MM2] max-mixture between  $X$ \vero{as} in MM1 and  \vero{an} inverse inverse Smith  max-stable process $Y$ with $\tau(h)=h/\sqrt{\sigma_Y}$.\\ 
\item[M1] A TEG max-stable process  $X$ specified in  MM1.
\item[M2] A Brown-Resnik max-stable process $X$ \vero{as} specified in  MM1.
\item[M3] An inverse Brown-Resnik max-stable process $Y$ \vero{as} specified in  MM1.

\item[M4] A Smith max-stable process $X$ \vero{as} specified in  M2.
\item[M5] A inverse Smith max-stable process $Y$ \vero{as} specified in  M2.
\end{description}
For all \vero{the considered models}, the margin distribution \vero{are} assumed to be unit Fr\'echet. Therefore \vero{it requires to}  transform the data set to Fr\'echet. Most of papers 
\vero{(see for example \cite{bacro2016flexible} and  \cite{wadsworth2012dependence}) use parametric transformations: they fit}  GEV parameters for each location separately and then transform  the 
data to Fr\'echet. In this study, we adopted \vero{a} non-parametric transformation by \vero{using the empirical c.d.f.}. For  censored composite likelihood procedures, we set $u=0.9$ and 
$\delta=\infty$.

\begin{table}[h]
\caption{Summary of the fitted models. the distance scale is kilometres. Composite likelihood procedure indicated by CL  and  LS-madogram  procedure indicated by LS with selection criteria SC are CLIC 
and \vero{MIC}, respectively.} 
\begin{tabular}[height=4.3cm]{l*{6}{c}r}
Model            &   & a & $\theta_X$& $r_X$ & $\theta_Y$ & $\sigma_Y$  & SC  \\
\hline
\hline
  MM1   &CL   &0.262  &1217.3 &1364.5 &3102.4 &  3.457  &{\bf 6807406}\\   
             &LS   &0.259  &1285.7& 1390.0& 5794.8  & 2.013&{\bf 1.917034}\\
\hline   
 MM2    &CL   & 0.248& 31.16 &70.15 &998.84 &&7924609 \\
            &LS  &0.185 &35.51 &48.14 &871.19&&1.917234\\
\hline
\hline
        &    &  $\theta_X$& $r_X$ & &   &&  \\   
\hline
    M1        &CL    & 931   &307.86&&&&7926261  \\
                 &LS    &1270 & 255.64&&&& 1.945177\\
\hline 
\hline 
        &    &  $\theta_X$& $\sigma_X$ &$\theta_Y$ &$\sigma_Y$   &&   \\   
\hline

    M2        &CL  & 931.02 & 3.078663 & &  &  &   7926261 \\
            &LS & 361.36 & 1.90816 & &  &  &1.96165\\
\hline   

    M3        &CL   &  &  &1644.76  &2.702282& &7918643  \\
            &LS  &  &  &  1383.08& 1.394928&  &1.924574 \\
\hline   

    M4        &CL   &  & 85.34  & &  & & 8016633  \\
            &LS  &  & 193.43 &  &  &  & 1.988753  \\
\hline   

    M5        &CL   &  &  &  &  & 256.39  & 7988838  \\
                 &LS  &  &  &  &  & 334.60&1.929235 \\
\hline      
\hline   
\end{tabular}
\end{table}
\vero{We remark that the two Decision Criteria (CLIC and MIC) choose the same model. We would like to emphasize that CLIC and MIC are not comparable, one is related to the composite likelihood while 
the other one is related to the least squared madogram difference. For the tested models, we would keep the model with the smallest CLIC or the smallest MIC, depending on the used estimation method. 
Also, the least squared madogram estimation is involves less computations, indeed, the maximum composite likelihood estimation requires to estimate the Godambe matrix in order to compute the CLIC.}
\section{Conclusions}\label{Se:7}
\vero{The calculation of the F-madogram for max-mixture processes show that it writes with both the AD and the AI parameters. This leads us to propose} a semi-parametric estimation procedure using 
F-madogram ${\nu}^F(h)$ as an alternative  to composite likelihood. The simulation study showed that the estimation procedure based on  $\nu^F(h)$ performs better than the composite likelihood 
procedure when the model is near to asymptotic independence. We applied these estimator procedures to real data example. \vero{On the considered example, the results obtained by composite likelihood 
maximization and least squared madogram difference are similar.} \\
\ \\
\underline{Acknowledgements:} This work was supported by the LABEX MILYON (ANR-10-LABX-0070) of Université de Lyon, within the program ``Investissements d'Avenir'' (ANR-11-IDEX-0007) operated by 
the French National Research Agency (ANR). We also acknowledge the projet LEFE CERISE.
\bibliographystyle{apalike}
\bibliography{sample.bib}

\begin{bibdiv}
\begin{biblist}

\bib{aghakouchak2010semi}{article}{
      author={AghaKouchak, A.},
      author={Nasrollahi, N.},
       title={Semi-parametric and parametric inference of extreme value models
  for rainfall data},
        date={2010},
     journal={Water resources management},
      volume={24},
      number={6},
       pages={1229\ndash 1249},
}

\bib{antoniadis1992regression}{book}{
      author={Antoniadis, A.},
      author={Berruyer, J.},
      author={Ren{\'e}, C.},
       title={R{\'e}gression non lin{\'e}aire et applications},
   publisher={Economica},
        date={1992},
}

\bib{bacro2010testing}{article}{
      author={Bacro, J-N.},
      author={Bel, L.},
      author={Lantu{\'e}joul, C.},
       title={Testing the independence of maxima: from bivariate vectors to
  spatial extreme fields},
        date={2010},
     journal={Extremes},
      volume={13},
      number={2},
       pages={155\ndash 175},
}

\bib{bacro2014estimation}{article}{
      author={Bacro, J-N.},
      author={Gaetan, C.},
       title={Estimation of spatial max-stable models using threshold
  exceedances},
        date={2014},
     journal={Statistics and Computing},
      volume={24},
      number={4},
       pages={651\ndash 662},
}

\bib{bacro2016flexible}{article}{
      author={Bacro, J-N.},
      author={Gaetan, C.},
      author={Toulemonde, G.},
       title={A flexible dependence model for spatial extremes},
        date={2016},
     journal={Journal of Statistical Planning and Inference},
      volume={172},
       pages={36\ndash 52},
}

\bib{bacro2013measuring}{article}{
      author={Bacro, J-N.},
      author={Toulemonde, G.},
       title={Measuring and modelling multivariate and spatial dependence of
  extremes},
        date={2013},
     journal={Journal de la Soci{\'e}t{\'e} Fran{\c{c}}aise de Statistique},
      volume={154},
      number={2},
       pages={139\ndash 155},
}

\bib{beirlant2006statistics}{book}{
      author={Beirlant, J.},
      author={Goegebeur, Y.},
      author={Segers, J.},
      author={Teugels, J.},
       title={Statistics of extremes: theory and applications},
   publisher={John Wiley \& Sons},
        date={2006},
}

\bib{buhl2016Semiparametric}{article}{
      author={Buhl, S.},
      author={Davis, R.~A.},
      author={Kl{\"u}ppelberg, C.},
      author={Steinkohl, C.},
       title={Semiparametric estimation for isotropic max-stable space-time
  processes},
        date={2016},
     journal={submitted},
}

\bib{buishand1984bivariate}{article}{
      author={Buishand, T.A.},
       title={Bivariate extreme-value data and the station-year method},
        date={1984},
     journal={Journal of Hydrology},
      volume={69},
      number={1},
       pages={77\ndash 95},
}

\bib{coles2001introduction}{book}{
      author={Coles, S.},
      author={Bawa, J.},
      author={Trenner, L.},
      author={Dorazio, P.},
       title={An introduction to statistical modeling of extreme values},
   publisher={Springer},
        date={2001},
      volume={208},
}

\bib{coles1999dependence}{article}{
      author={Coles, S.},
      author={Heffernan, J.},
      author={Tawn, J.A.},
       title={Dependence measures for extreme value analyses},
        date={1999},
     journal={Extremes},
      volume={2},
      number={4},
       pages={339\ndash 365},
}

\bib{cooley2006variograms}{incollection}{
      author={Cooley, D.},
      author={Naveau, P.},
      author={Poncet, P.},
       title={Variograms for spatial max-stable random fields},
        date={2006},
   booktitle={Dependence in probability and statistics},
   publisher={Springer},
       pages={373\ndash 390},
}

\bib{davis2011comments}{article}{
      author={Davis, R.~A.},
      author={Yau, C.~Y.},
       title={Comments on pairwise likelihood in time series models},
        date={2011},
     journal={Statistica Sinica},
       pages={255\ndash 277},
}

\bib{davison2012geostatistics}{inproceedings}{
      author={Davison, A.~C.},
      author={Gholamrezaee, M.},
       title={Geostatistics of extremes},
organization={The Royal Society},
        date={2012},
   booktitle={Proceedings of the royal society of london a: Mathematical,
  physical and engineering sciences},
      volume={468},
       pages={581\ndash 608},
}

\bib{davison2013geostatistics}{article}{
      author={Davison, A.C.},
      author={Huser, R.},
      author={Thibaud, E.},
       title={Geostatistics of dependent and asymptotically independent
  extremes},
        date={2013},
     journal={Mathematical Geosciences},
      volume={45},
      number={5},
       pages={511\ndash 529},
}

\bib{de1984spectral}{article}{
      author={De~Haan, L.},
       title={A spectral representation for max-stable processes},
        date={1984},
     journal={The annals of probability},
       pages={1194\ndash 1204},
}

\bib{de2006spatial}{article}{
      author={De~Haan, L.},
      author={Pereira, T.~T.},
       title={Spatial extremes: Models for the stationary case},
        date={2006},
     journal={The annals of statistics},
       pages={146\ndash 168},
}

\bib{carlstein1986use}{article}{
      author={Edward, C.},
       title={The use of subseries values for estimating the variance of a
  general statistic from a stationary sequence},
        date={1986},
     journal={The Annals of Statistics},
       pages={1171\ndash 1179},
}

\bib{fonseca2011generalized}{article}{
      author={Fonseca, C.},
      author={Pereira, L.},
      author={Ferreira, H.},
      author={Martins, A.P.},
       title={Generalized madogram and pairwise dependence of maxima over two
  regions of a random field},
        date={2015},
     journal={KYBERNETIKE},
      volume={51},
      number={2},
       pages={193\ndash 211},
}

\bib{guillou2014madogram}{article}{
      author={Guillou, A.},
      author={Naveau, P.},
      author={Schorgen, A.},
       title={Madogram and asymptotic independence among maxima},
        date={2014},
     journal={REVSTAT--Statistical Journal},
      volume={12},
      number={2},
       pages={119\ndash 134},
}

\bib{guyon1995random}{book}{
      author={Guyon, X.},
       title={Random fields on a network: modeling, statistics, and
  applications},
   publisher={Springer Science \& Business Media},
        date={1995},
}

\bib{kabluchko2009stationary}{article}{
      author={Kabluchko, Z.},
      author={Schlather, M.},
      author={De~Haan, L.},
       title={Stationary max-stable fields associated to negative definite
  functions},
        date={2009},
     journal={The Annals of Probability},
       pages={2042\ndash 2065},
}

\bib{ledford1996statistics}{article}{
      author={Ledford, A.W.},
      author={Tawn, J.A.},
       title={Statistics for near independence in multivariate extreme values},
        date={1996},
     journal={Biometrika},
      volume={83},
      number={1},
       pages={169\ndash 187},
}

\bib{lindsay1888composite}{article}{
      author={Lindsay, B.G.},
       title={Composite likelihood methods},
        date={1988},
     journal={Contemporary Math.},
       pages={221\ndash 239},
}

\bib{naveau2009modelling}{article}{
      author={Naveau, P.},
      author={Guillou, A.},
      author={Cooley, D.},
      author={Diebolt, J.},
       title={Modelling pairwise dependence of maxima in space},
        date={2009},
     journal={Biometrika},
      volume={96},
      number={1},
       pages={1\ndash 17},
}

\bib{northrop2015efficient}{article}{
      author={Northrop, P.~J.},
       title={An efficient semiparametric maxima estimator of the extremal
  index},
        date={2015},
     journal={Extremes},
      volume={18},
      number={4},
       pages={585\ndash 603},
}

\bib{padoan2010likelihood}{article}{
      author={Padoan, S.~A.},
      author={Ribatet, M.},
      author={Sisson, S.~A.},
       title={Likelihood-based inference for max-stable processes},
        date={2010},
     journal={Journal of the American Statistical Association},
      volume={105},
      number={489},
       pages={263\ndash 277},
}

\bib{schlather2002models}{article}{
      author={Schlather, M.},
       title={Models for stationary max-stable random fields},
        date={2002},
     journal={Extremes},
      volume={5},
      number={1},
       pages={33\ndash 44},
}

\bib{schlather2003dependence}{article}{
      author={Schlather, M.},
      author={Tawn, J.A.},
       title={A dependence measure for multivariate and spatial extreme values:
  Properties and inference},
        date={2003},
     journal={Biometrika},
      volume={90},
      number={1},
       pages={139\ndash 156},
}

\bib{sibuya1959bivariate}{article}{
      author={Sibuya, M.},
       title={Bivariate extreme statistics, i},
        date={1959},
     journal={Annals of the Institute of Statistical Mathematics},
      volume={11},
      number={2},
       pages={195\ndash 210},
}

\bib{smith1990max}{article}{
      author={Smith, R.~L.},
       title={Max-stable processes and spatial extremes},
        date={1990},
     journal={Unpublished manuscript, Univer},
}

\bib{thibaud2013threshold}{article}{
      author={Thibaud, E.},
      author={Mutzner, R.},
      author={Davison, A.C.},
       title={Threshold modeling of extreme spatial rainfall},
        date={2013},
     journal={Water resources research},
      volume={49},
      number={8},
       pages={4633\ndash 4644},
}

\bib{varin2011overview}{article}{
      author={Varin, C.},
      author={Reid, N.},
      author={Firth, D.},
       title={An overview of composite likelihood methods},
        date={2011},
     journal={Statistica Sinica},
       pages={5\ndash 42},
}

\bib{varin2005note}{article}{
      author={Varin, C.},
      author={Vidoni, P.},
       title={A note on composite likelihood inference and model selection},
        date={2005},
     journal={Biometrika},
       pages={519\ndash 528},
}

\bib{wackernagel1998multivariate}{incollection}{
      author={Wackernagel, H.},
       title={Multivariate nested variogram},
        date={1998},
   booktitle={Multivariate geostatistics},
   publisher={Springer},
       pages={172\ndash 180},
}

\bib{wadsworth2012dependence}{article}{
      author={Wadsworth, J.L.},
      author={Tawn, J.A.},
       title={Dependence modelling for spatial extremes},
        date={2012},
     journal={Biometrika},
      volume={99},
      number={2},
       pages={253\ndash 272},
}

\bib{zheng2015assessing}{article}{
      author={Zheng, F.},
      author={Thibaud, E.},
      author={Leonard, M.},
      author={Westra, S.},
       title={Assessing the performance of the independence method in modeling
  spatial extreme rainfall},
        date={2015},
     journal={Water Resources Research},
      volume={51},
      number={9},
       pages={7744\ndash 7758},
}

\bib{zheng2014modeling}{article}{
      author={Zheng, F.},
      author={Westra, S.},
      author={Leonard, M.},
      author={Sisson, S.~A.},
       title={Modeling dependence between extreme rainfall and storm surge to
  estimate coastal flooding risk},
        date={2014},
     journal={Water Resources Research},
      volume={50},
      number={3},
       pages={2050\ndash 2071},
}

\end{biblist}
\end{bibdiv}

\end{document}